\documentclass[11 pt]{amsart}

\usepackage{amsrefs}

\theoremstyle{plain}
\newtheorem{theorem}{Theorem}
\newtheorem{proposition}[theorem]{Proposition}

\newtheorem{lemma}[theorem]{Lemma}

\theoremstyle{remark}
\newtheorem*{notations}{Notations and conventions}
\newtheorem{remark}[theorem]{Remark}

\theoremstyle{definition}
\newtheorem{definition}[theorem]{Definition}

\numberwithin{equation}{section}

\DeclareMathOperator{\Aut}{Aut}
\DeclareMathOperator{\End}{End}
\DeclareMathOperator{\Res}{Res}
\DeclareMathOperator{\St}{St}
\DeclareMathOperator{\Tr}{Tr}

\newcommand{\bZ}{\mathbb Z}
\newcommand{\bQ}{\mathbb Q}
\newcommand{\bR}{\mathbb R}
\newcommand{\bC}{\mathbb C}

\newcommand{\sA}{\mathcal A}
\newcommand{\sB}{\mathcal B}

\newcommand{\tU}{\widetilde U}
\newcommand{\tV}{\widetilde V}
\newcommand{\tW}{\widetilde W}

\newcommand{\pel}{\textsc{pel}}

\begin{document}

\title[Tate twists of Hodge structures]{Tate twists of Hodge structures arising from abelian varieties of type IV}

\author{Salman Abdulali}
\address{Department of Mathematics, East Carolina University, Greenville, NC 27858, USA}
\email{abdulalis@ecu.edu}


\begin{abstract}
We show that certain abelian varieties $A$ have the property that for every Hodge structure $V$ in the cohomology of $A$, every effective Tate twist of $V$ occurs in the cohomology of some abelian variety. We deduce the general Hodge conjecture for certain non-simple abelian varieties of type~IV.
\end{abstract}

\maketitle

\section{Introduction}

A (rational) Hodge structure
$V_{\bC} = \bigoplus_{p+q=n} V^{p,q}$
is said to be \emph{effective} if $V^{p,q} = 0$ unless $p$,~$q \geq 0$, and,
it is said to be \emph{geometric} (or \emph{motivic}) if it is isomorphic to a Hodge substructure of $H^n(X, \bQ)$ for some smooth, projective variety $X$ over $\bC$.
For $m \in \bZ$, the Tate twist $V(m)$ is the Hodge structure of weight $n-2m$ defined by
$V(m)^{p,q} = V^{p+m,q+m}$.

A geometric Hodge structure must be effective and polarizable, but not conversely (Grothendieck \cite{Grothendieck}*{p.~300, 2nd footnote}). It is well-known that any polarizable Hodge structure of weight $1$ is the first cohomology of an abelian variety, and hence geometric. In \cite{Abdulali2005} we have shown that any Hodge structure of CM-type is geometric. These are the only known criteria for an abstract Hodge structure to be geometric \cite{GGK}*{p.~305}.

The general Hodge conjecture as formulated by Grothendieck \cite{Grothendieck} implies that any effective Tate twist of a geometric Hodge structure is again geometric.
In a series of papers \cites{Abdulali1997, Abdulali1999, Abdulali2000, type3ghc, Abdulali2004, Abdulali2005} we have shown that, for certain abelian varieties $A$, every effective Tate twist of a Hodge structure in the cohomology of $A$ is isomorphic to a Hodge structure occurring in the cohomology of some abelian variety. Moreover, we have used this to prove the general Hodge conjecture for certain abelian varieties. We have also shown the existence of a Hodge structure which occurs in the cohomology of an abelian variety, but which has an effective Tate twist that does not occur in the cohomology of \emph{any} abelian variety \cite{type3ghc}*{Theorem~5.5, p.~926}.

Our earlier results apply to abelian varieties of type IV in only very special cases (see \S \ref{groups} for the definition of the type of an abelian variety) --- namely when the Hodge group is semisimple \cite{Abdulali1997}, or when the abelian variety is of CM-type \cite{Abdulali2005}, or when the endomorphism algebra is an imaginary quadratic number field \cite{Abdulali2004}.
The main aim of this paper is to remove these restrictions on the endomorphism algebra; however, we still need a fairly strong restriction on the signature of the hermitian form determining the polarization; see Theorem \ref{maintheorem} for the precise statement.
As an application of these results we deduce the general Hodge conjecture for \emph{products} of some abelian varieties of type IV (Theorem~\ref{ghc}).

\begin{notations}
All abelian varieties are over $\bC$.
Representations are always finite dimensional.
For an abelian variety $A$, we let
\[
	D(A) = \End_{\bQ} (A) := \End (A) \otimes \bQ
\]
be its endomorphism algebra, $L(A)$ its Lefschetz group, $G(A)$ its Hodge group, and, $G'(A)$ the derived group of $G(A)$; see \S \ref{groups} for more details.
\end{notations}

\section{Hodge groups and Lefschetz groups}
\label{groups}

Let $A$ be an abelian variety over $\bC$, and let $V = H^1(A, \bQ)$.
The Hodge group $G(A)$ is defined in Mumford \cite{Mumford1966}.
It is the reductive $\bQ$-algebraic subgroup of $GL(V)$ characterized by the property that its invariants in $H^{\star}(A^n, \bQ)$ are precisely the Hodge classes for any positive integer $n$.

The Lefschetz group $L(A)$ is defined in Murty \cite{murtybanff}*{\S 3.6.2, p.~93}.
It is the reductive $\bQ$-algebraic subgroup of $GL(V)$ characterized by the property that for any positive integer $n$, its invariants in $H^{\star}(A^n, \bQ)$ form the ring generated by divisor classes.
Since any divisor class is a Hodge class, it follows that $G(A) \subset L(A)$.
Note that the \lq\lq Lefschetz group\rq\rq\ defined by Murty in  \cite{Murty} is the connected component of the identity in the group  defined as the Lefschetz group in \cite{murtybanff}.

We say that $A$ is of \emph{\pel-type} if the semisimple parts of $G(A)$ and $L(A)^0$ are equal.
A simple abelian variety is of \pel-type if and only if it is a general member of a \pel-family of abelian varieties (see \cite{Abdulali1997}*{\S1 and \S4.6}).

Suppose $A$ is a simple abelian variety.
Let $\beta$ be an alternating Riemann form for $A$.
Let $D = D(A) = \End (A) \otimes \bQ$ be its endomorphism algebra.
By Albert's classification, $D$ is one of the following \cite{Shimura1}:
\begin{description}
	\item[type I]
		a totally real number field $F$
	\item[type II]
		a totally indefinite quaternion algebra over a totally real number field $F$
	\item[type III]
		a totally definite quaternion algebra over a totally real number field $F$
	\item[type IV]
		a division algebra over a CM-field $E$.
		In this case let $F$ be the maximal totally real subfield of $E$.
\end{description}
In each case there exists an involution $x \mapsto \overline{x}$ of $D$, and a unique $F$-bilinear form $T \colon V \times V \to D$ such that $\beta (x,y) = \Tr_{D/\bQ} T(x,y)$, $T(ax,by) = aT(x,y)\overline{b}$, and, $T(y,x) = -\overline{T(x,y)}$ for all $x$,~$y \in V$, $a$,~$b \in D$ \cite{Shimura2}*{Lemma~1.2, p.~162}.
The Lefschetz group is then the restriction of scalars, from $F$ to $\bQ$ of the unitary group of $T$:
\begin{equation}
\label{lefschetzproduct}
	L(A) = \Res_{F/\bQ} U(T) = \Res_{F/\bQ} \Aut_D (V, T).
\end{equation}

Let $S$ be the set of embeddings of $F$ into $\bR$.
We can then write
\begin{equation}
\label{Sdecomp}
	L(A)_{\bR} = \prod_{\alpha \in S} L_{\alpha} \qquad \text{and} \qquad V_{\bR} = \bigoplus_{\alpha \in S} V_{\alpha},
\end{equation}
where $L_{\alpha}$ acts trivially on $V_{\alpha'}$ unless $\alpha = \alpha'$.
$L_{\alpha}$ and its action on $V_{\alpha}$ are given as follows \cite{Murty}:
\begin{description}
	\item[type I]
		$L_{\alpha} = Sp(V_{\alpha}, \beta_{\alpha})$ is a symplectic group
		acting via its standard representation on $V_{\alpha}$.
	\item[type II]
		$L_{\alpha}$ is a symplectic group acting on $V_{\alpha}$
		as two copies of the standard representation.
	\item[type III]
		$L_{\alpha,\bC}$ is an orthogonal group acting on $V_{\alpha,\bC}$
		as two copies of the standard representation.
	\item[type IV]
		$L_{\alpha} = U(p_{\alpha}, q_{\alpha})$, and
		$L_{\alpha,\bC} \cong GL_m(\bC)$ acts on $V_{\alpha,\bC}$
		as the direct sum of the standard representation and its contragredient.
\end{description}

\section{Dominating Varieties}

We say that a Hodge structure $V$ is \emph{fully twisted} if $V$ is effective, but the Tate twist $V(1)$ is not effective.
Thus $V_{\bC} = \bigoplus_{p+q=n} V^{p,q}$ is fully twisted if and only if it is effective and $V^{n,0} \neq 0$.

We say that a smooth, projective algebraic variety $A$ over $\bC$ is \emph{dominated} by a class $\sA$ of smooth, projective complex algebraic varieties if, given any irreducible Hodge structure $V$ in the cohomology of $A$, there exists a fully twisted Hodge structure $V'$ in the cohomology of some $X \in \sA$ such that $V'$ is isomorphic to a Tate twist of $V$.

\begin{proposition} [Grothendieck \cite{Grothendieck}*{p.~301}]
\label{groprop}
Let $A$ be a smooth projective variety over $\bC$ which is dominated by $\sA$.
If the usual Hodge conjecture holds for $A \times B$ for each $B \in \sA$, then the general Hodge conjecture holds for~$A$.
\end{proposition}

\begin{proof}[Reference to proof]
See the proof of \cite{Abdulali1997}*{Proposition~2.1, p.~343}.
\end{proof}

Let $A$ be an abelian variety.
Let $k$ be a subfield of $\bC$.
Let $\sA$ be a class of abelian varieties.
We say that $A$ is $k$-\emph{dominated} by $\sA$ if, given any irreducible $G(A)_k$-submodule, $V$, of $H^{\star}(A,k)$, there exist $B \in \sA$, and a $G(B)_k$-submodule $V'$ of $H^n(B,k)$ for some $n$, such that $V$ and $V'$ are isomorphic as $G(A \times B)_k$-modules, and, $V'_{\bC}$ contains a nonzero $(n,0)$-form.
(Note that $G(A \times B)$ is a subgroup of $G(A) \times G(B)$, so it makes sense to consider $V$ and $V'$ as $G(A \times B)_k$-modules.)
In particular, $A$ is dominated by $\sA$ if and only if $A$ is $\bQ$-dominated by $\sA$.

\begin{lemma}
\label{fieldlemma}
If an abelian variety $A$ is $k$-dominated by $\sA$ for some subfield $k$ of $\bC$, then, $A$ is dominated by $\sA$.
\end{lemma}

\begin{proof}
Let $W$ be an irreducible Hodge structure in the cohomology of $A$.
Then $W$ is an irreducible $G(A)$-module.
Let $W_0$ be an irreducible $G(A)_k$-submodule of $W_k$.
Then there exist $B \in \sA$, and, a $G(B)_k$-submodule $W'_0$ of $H^n(B,k)$ such that
$W_0$ and $W'_0$ are isomorphic as $G(A \times B)_k$-modules, and,
$W'_{0,\bC}$ contains a nonzero $(n,0)$-form.

Let $G = G(A \times B)$.
Since $W_0$ and $W'_0$ are isomorphic as $G_k$-modules, $\hom_{G_k} (W_k, H^n (B,k))$ is nontrivial.
Since $G(\bQ)$ is Zariski-dense in $G(k)$, we have
\[
\hom_{G_k} (W_k, H^n (B,k)) = \hom_G (W, H^n (B, \bQ)) \otimes k.
\]
Thus $\hom_G (W, H^n (B, \bQ))$ contains a nonzero element $\varphi$
such that $\varphi(W_0) = W'_0$.
Let $W'$ be the image of $\varphi$.
Then $W'$ is fully twisted because $W'_{0,\bC} \subset W'_{\bC}$.
Since $W$ is irreducible, $\varphi$ must be a $G$-isomorphism from $W$ to $W'$.
This means that as Hodge structures, $W$ and $W'$ are isomorphic up to a Tate twist.
\end{proof}

\begin{proposition}
\label{productlemma}
Let $A$ and $B$ be abelian varieties such that $G(A \times B) = G(A) \times G(B)$.
If $A$ is $\bC$-dominated by $\sA$, and $B$ is $\bC$-dominated by 
$\sB$, then, $A \times B$ is $\bC$-dominated by
$\sA \cdot \sB = \{ \, X \times Y \mid X \in \sA, \ Y \in \sB \, \}$.
\end{proposition}

\begin{proof}
Let $W \subset H^n(A \times B,\bC)$ be an irreducible $G(A)_{\bC} \times G(B)_{\bC}$-module.
Then $W$ is contained in a K\"{u}nneth component
$H^a(A, \bC) \otimes H^b(B, \bC)$ with $a+b=n$.
So we can write $W = U \otimes V$, where $U \subset H^a(A, \bC)$ and $V \subset H^b(B, \bC)$ are irreducible $G(A)_{\bC}$ and $G(B)_{\bC}$ modules, respectively.
By assumption there exist abelian varieties $X \in \sA$ and $Y \in \sB$,
a $G(X)_{\bC}$-submodule $U'$ of $H^m (X,\bC)$, and, a $G(Y)_{\bC}$-submodule $V'$ of $H^n (Y,\bC)$, such that $U'$ and $V'$ contain nonzero $(m,0)$ and $(n,0)$ forms respectively, $U'$ is $G(A \times X)_{\bC}$-isomorphic to $U$, and, $V'$ is $G(B \times Y)_{\bC}$-isomorphic to $V$.

Let $W' = U' \otimes V' \subset H^{m+n}(X \times Y, \bC)$.
Since $G(X \times Y)$ is a subgroup of $G(X) \times G(Y)$, we see that
$W'$ is a $G(X \times Y)_{\bC}$-submodule of $H^{m+n}(X \times Y, \bC)$.
Clearly, it contains a nonzero $(m+n,0)$-form.
Since $G(A \times B \times X \times Y)$ is a subgroup of $G(A \times X) \times G(B \times Y)$,
$U$ is isomorphic to $U'$ as a $G(A \times X)_{\bC}$-module, and,
$V$ is isomorphic to $V'$ as a $G(B \times Y)_{\bC}$-module, we see that
$W$ is isomorphic to $W'$ as a $G(A \times B \times X \times Y)_{\bC}$-module.
\end{proof}

\begin{proposition}
\label{weakproductlemma}
Let $A$ and $B$ be abelian varieties such that $G(A \times B) = G(A) \times G(B)$.
If $A$ is $\bC$-dominated by $\sA$, and $B$ is dominated by $\sB$, then, $A \times B$ is dominated by
$\sA \cdot \sB = \{ \, X \times Y \mid X \in \sA, \ Y \in \sB \, \}$.
\end{proposition}

\begin{proof}
Let $W \subset H^n(A \times B,\bQ)$ be an irreducible Hodge structure.
Then $W$ is contained in a K\"{u}nneth component
$H^a(A, \bQ) \otimes H^b(B, \bQ)$ with $a+b=n$.
Let $W_0$ be an irreducible $G(A \times B)_{\bC}$-submodule of $W_{\bC}$.
Write $W_0 = U_0 \otimes V_0$, where $U_0 \subset H^a(A, \bC)$ is an irreducible $G(A)_{\bC}$-module, and,
$V_0 \subset H^b(B, \bC)$ is an irreducible $G(B)_{\bC}$-module.

Let $\tV \subset H^b (B, \bQ)$ be the smallest Hodge structure such that $V_0 \subset \tV_{\bC}$;
it is the sum of all the Galois conjugates of $V_0$.
Then $\tV$ is a primary $G(B)$-module, i.e., all irreducible submodules of $\tV$ are equivalent.
Let $\tV_1$ be an irreducible submodule of $\tV$;
then $\tV_{1,\bC}$ contains a $G(B)_{\bC}$-submodule $V_1$ equivalent to $V_0$.
Since $B$ is dominated by $\sB$, there exist $Y \in \sB$, and $\tV' \subset H^d(Y, \bQ)$ such that $\tV'$ is $G(B \times Y)$-equivalent to $\tV_1$, and, $\tV'$ contains a nonzero $(d,0)$-form.
Let $V'_1$ be an irreducible $G(Y)_{\bC}$-submodule of $\tV'_{\bC}$ such that $V'_1$ contains a nonzero $(d, 0)$-form.
We have $V'_1$ equivalent, as a $G(B \times Y)_{\bC}$-module, to a Galois conjugate $V_1^{\sigma}$ of $V_1$
for some $\sigma \in \Aut (\bC)$.
Then $W_0^{\sigma}$ and  $U_0^{\sigma} \otimes V'_1$
are equivalent as $G(A \times B \times Y)_{\bC}$-modules.

Since $A$ is $\bC$-dominated by $\sA$, there exist $X \in \sA$,
and $U'_0 \subset H^c (X,\bC)$ such that $U'_0$ is
$G(A \times X)_{\bC}$-equivalent to $U_0^{\sigma}$, and, $U'_0$ contains a nonzero $(c, 0)$-form.
Then $W_0^{\sigma}$ and $U'_0 \otimes V'_1$  are equivalent
as $G(A \times B \times X \times Y)_{\bC}$-modules.

Now $U'_0 \otimes V'_1$ is an irreducible $G(X \times Y)_{\bC}$-submodule of
$H^{c+d} (X \times Y,\bC)$ which contains a nonzero $(c+d, 0)$-form.
Let $\tW$ be the smallest Hodge structure containing $U'_0 \otimes V'_1$.
Then, $\tW$ is a primary $G(X \times Y)$-module, and any irreducible Hodge substructure of $\tW$ is fully twisted.
Since $W_0^{\sigma} \subset W_{\bC}$ and $U'_0 \otimes V'_1 \subset \tW_{\bC}$ are equivalent as
$G(A \times B \times X \times Y)_{\bC}$-modules, $W$ is equivalent to an irreducible Hodge substructure $W'$ of $\tW$.
This completes the proof since $W'$ is fully twisted.
\end{proof}

Propositions \ref{productlemma} and \ref{weakproductlemma} replace
Proposition~4.4.1 of \cite{Abdulali1997} which contains an error
(the first sentence of the proof is not correct in general).
We now reformulate part of the main theorem of \cite{Abdulali1997}.
Abelian varieties of type III are excluded here; they have been dealt with in \cite{type3ghc}.

\begin{theorem}
\label{semisimple}
Let $A$ be an abelian variety of \pel-type.
Suppose that the Hodge group of $A$ is semisimple and $A$ has no factors of type \textup{III}.
Then $A$ is $\bC$-dominated by the set of powers of itself.
The usual Hodge conjecture for $A$ implies the general Hodge conjecture for all powers of $A$.
\end{theorem}

\begin{proof}[Sketch of proof]
$A$ is isogenous to a product $A_1^{n_1} \times A_2^{n_2} \times \dots \times A_{\ell}^{n_{\ell}}$
where the $A_i$ are pairwise nonisogenous abelian varieties.
By the multiplicativity of the Lefschetz group (Murty \cite{Murty}*{Lemma 2.1, p.~298}), we have
\[
	L(A) = L(A_1) \times L(A_2) \times \dots \times L(A_{\ell}).
\]
Since
\[
	G(A) \subset G(A_1) \times G(A_2) \times \dots \times G(A_{\ell})
\]
and $G(A)$ equals the derived group of $L(A)$, we conclude that each $A_i$ is of \pel-type, and,
\[
	G(A) = G(A_1) \times G(A_2) \times \dots \times G(A_{\ell}).
\]
Lemma~\ref{fieldlemma} and Proposition~\ref{productlemma} now imply that it is enough to prove the theorem when $A$ is a power of a simple abelian variety $A_0$.

Let $G = G(A) = G(A_0)$, let $D$ be the endomorphism algebra of $A_0$, $E$ the center of $D$, and $F$ the maximal real subfield of $E$.
Let $S$ be the set of embeddings of $F$ into $\bR$.
From \eqref{Sdecomp} we see that $G(\bR) = \prod_{\alpha \in S} G_{\alpha}$, and,
\(
	H^1(A_0,\bR) = \bigoplus_{\alpha \in S} V_{\alpha},
\)
where each $V_{\alpha}$ is a real Hodge substructure of  $H^1(A_0,\bR)$ on which
$G_{\gamma}$ acts trivially for $\gamma \neq \alpha$.

Now let $W$ be any irreducible $G_{\bC}$-submodule of the cohomology of $A$.
Then $W$ is equivalent to a representation $\bigotimes_{\alpha \in S} W_{\alpha}$, where
$W_{\alpha}$ is an irreducible representation of $G_{\alpha,\bC}$.
In Cases 1 and 2 of the proof of \cite{Abdulali1997}*{Theorem~5.1} we 
showed that there exist $G_{\alpha,\bC}$-submodules
\[
	W'_{\alpha} \subset
	\bigwedge^{n_{\alpha}} V_{\alpha,\bC}^{m_{\alpha}} \subset
	H^{n_{\alpha}} (A_0^{m_{\alpha}},\bC)
\]
for some $n_{\alpha}$, $m_{\alpha}$, such that $W_{\alpha}$ and 
$W'_{\alpha}$ are equivalent, and, $W'_{\alpha}$ contains a nonzero
$(n_{\alpha}, 0)$-form.
Let $n = \sum_{\alpha} n_{\alpha}$, $m = \sum_{\alpha} m_{\alpha}$, 
and, $W' = \bigotimes_{\alpha} W'_{\alpha}$.
Then, $W' \subset \bigwedge^n V_{\bC}^m = H^n (A_0^m, \bC)$ is 
equivalent to $W$ and contains a nonzero $(n,0)$-form.
This shows that $A$ is $\bC$-dominated by the set of powers of  itself.

To complete the proof we remark that the usual Hodge conjecture for $A$
implies the usual Hodge conjecture for all powers of $A$.
This follows from \cite{Abdulali1999}*{Theorem~3.1, p.~671}.
\end{proof}

\begin{remark}
In \cites{type3ghc, Abdulali2004, Abdulali2005} we have proved the general Hodge conjecture for various abelian varieties which are dominated, but not $\bC$-dominated, by certain classes of abelian varieties.
Proposition \ref{weakproductlemma} allows us to deduce the general Hodge conjecture for the product of one of these abelian varieties with an abelian variety satisfying the hypotheses of Theorem \ref{semisimple}.
\end{remark}

\section{Abelian Varieties of Type IV}

Let $A$ be an abelian variety of type IV.
If $G(A)$ is semisimple, then we have seen (Theorem~\ref{semisimple}) that $A$ is $\bC$-dominated by powers of itself.
At the other extreme, if $G(A)$ is commutative, then, $A$ is of CM-type, and we have shown in \cite{Abdulali2005} that $A$ is dominated by abelian varieties of CM-type.
We shall now extend these results to some abelian varieties of type IV whose Hodge groups are neither semisimple nor commutative. We begin with a definition.

\begin{definition}
\label{selfdominated}
We say that an abelian variety $A$ is \emph{weakly self-dominated} if, given any nontrivial irreducible representation $\rho$ of $G'(A)(\bC)$, there exists $V_{\rho}$ such that
\begin{itemize}
	\item
		$V_{\rho}$ is an $L(A)(\bC)$-submodule of $H^{c_{\rho}}(A^{d_{\rho}}, \bC)$
		for some positive integers $c_{\rho}$, $d_{\rho}$;
	\item
		the action of  $G'(A)(\bC)$ on $V_{\rho}$ is equivalent to $\rho$;
	\item
		for each $\sigma \in \Aut (\bC)$, the conjugate $(V_{\rho})^{\sigma}$
		contains a nonzero $(c_{\rho},0)$-form.
\end{itemize}
\end{definition}

\begin{remark}
In Theorem~\ref{mainexample} below, we show that certain type IV abelian varieties of \pel-type are weakly self-dominated.
In Theorem~\ref{weakimpliesdomination} we show that if $A$ is weakly self-dominated, then $A$ is dominated by abelian varieties of the form $A^n \times B$ where $B$ is of CM-type.
In Theorem~\ref{ghc} we apply these results to prove the general Hodge conjecture for some of these abelian varieties.
\end{remark}

\begin{lemma}
\label{selfdominationlemma}
Any abelian variety of CM-type is weakly self-dominated.
If $A$ is weakly self-dominated, then, so is any power of $A$.
If $A$ and $B$ are weakly self-dominated abelian varieties such that
$G'(A \times B) = G'(A) \times G'(B)$, then $A \times B$ is also weakly self-dominated.
\end{lemma}

\begin{proof}
The first statement is trivial.
The second statement is immediate from the definition.
For the third statement, note that any irreducible representation of $G'(A \times B)(\bC)$ is of the form $\rho \otimes \tau$, where $\rho$ is an irreducible representation of $G'(A)(\bC)$ and $\tau$ is an irreducible representation of $G'(B)(\bC)$.
Let
\[
	V_{\rho \otimes \tau} =
		\begin{cases}
			V_{\rho} \otimes V_{\tau} &\text{if both $\rho$ and $\tau$ are nontrivial;} \\
			V_{\rho} &\text{if $\rho$ is nontrivial but $\tau$ is trivial;}\\
			V_{\tau} &\text{if $\tau$ is nontrivial but $\rho$ is trivial.}
		\end{cases}
\]
\end{proof}

\begin{theorem}
\label{mainexample}
Let $A$ be an abelian variety of \pel-type such that each simple factor of $A$ is of type \textup{IV}.
Then we can write $G'(A)(\bR) \cong \prod_{\alpha \in S} SU(p_{\alpha}, q_{\alpha})$.
Assume that for each $\alpha \in S$ we have $\vert p_{\alpha} - q_{\alpha} \vert = 1$.
Then $A$ is weakly self-dominated.
\end{theorem}

\begin{proof}
Thanks to Lemma \ref{selfdominationlemma}, we may assume that $A$ is simple.
Let $L = L(A)$, $G = G(A)$, $G' = G'(A)$, and, $V = H^1 (A, \bQ)$.
Recall from \eqref{lefschetzproduct} that $L(A) = \Res_{F/{\bQ}} U(T)$,
where $U(T)$ is a unitary group over $F$, the maximal totally real subfield of the center $E$ of $D(A)$.
Let $S$ be the set of embeddings of $F$ into $\bR$.
Then \eqref{Sdecomp} we have $L(\bR) = \prod_{\alpha \in S} L_{\alpha}$, and,
$V_{\bR} = \bigoplus_{\alpha \in S} V_{\alpha}$,
so that $L_{\alpha}$ acts trivially on $V_{\alpha'}$ unless $\alpha = \alpha'$.
Each $L_{\alpha}$ is a unitary group $U(p_{\alpha}, q_{\alpha})$, with
$p_{\alpha} + q_{\alpha} = m := \dim_E V$.
If $m=1$, then $A$ is of CM-type, $G'(A)$ is trivial, and there is nothing to prove; we may therefore assume $m \geq 3$, so that all $p_{\alpha}$ and $q_{\alpha}$ are positive.

We have $L_{\alpha, \bC} \cong GL_m(\bC)$.
As explained in \cite{Abdulali1997}*{p.~351},
$V_{\alpha, \bC} = Y_{\alpha} \oplus \overline{Y}_{\alpha}$, where $Y_{\alpha}$ and its complex conjugate $\overline{Y}_{\alpha}$ are  $L_{\alpha, \bC}$-modules,
$GL_m(\bC)$ acts on $Y_{\alpha}$ as the standard representation, and on $\overline{Y}_{\alpha}$ as the contragredient.
$Y_{\alpha}$ is the direct sum of a $p_{\alpha}$-dimensional space of $(1,0)$-forms and a $q_{\alpha}$-dimensional space of $(0,1)$-forms.
 $\overline{Y}_{\alpha}$ is the direct sum of a $q_{\alpha}$-dimensional space of $(1,0)$-forms and a $p_{\alpha}$-dimensional space of $(0,1)$-forms.
Choose a basis $\left\{ u_1, \dotsc, u_m \right\}$ of $Y_{\alpha}$ such that
$u_1$, \dots, $u_{p_{\alpha}}$ are $(1,0)$-forms and $u_{p_{\alpha}+1}$, \dots, $u_m$ are $(0,1)$-forms.
Then $\left\{ \overline{u}_1, \dots, \overline{u}_m \right\}$ is a basis of $\overline{Y}_{\alpha}$.
Observe that the set $\bigcup_{\alpha \in S} \left\{ Y_{\alpha}, \overline{Y}_{\alpha} \right\}$ is invariant under the action of $\Aut (\bC)$.

Let $g$ be the element of $GL_m(\bC)$ which transposes $u_k$ and $u_{m-k+1}$ for each $k$.

Let $\mu_1$, \dots, $\mu_{m-1}$ be the fundamental weights of $SL_{m}(\bC)$, i.e.,
$\mu_k$ is the highest weight of the representation $\bigwedge^k (\St)$, where $(\St)$ denotes the standard representation of $SL_{m}(\bC)$ on $\bC^m$.
For $1 \leq k < \frac{m}{2}$, $V_{\alpha, k} := \bigwedge^k Y_{\alpha} \subset H^k(A,\bC)$
is an $L_{\alpha, \bC}$-module;
it is irreducible as a $G'_{\alpha,\bC}$-module, and has highest weight $\mu_k$.
It contains the $(k,0)$-form
\[
	w_k := u_1 \wedge \dots \wedge u_k,
\]
as well as the $(0,k)$-form
\[
	w'_k := g(w_k) = u_m \wedge \dots \wedge u_{m-k+1}.
\]

For $\frac{m}{2} < k < m$,  $V_{\alpha, k} := \bigwedge^{m-k} \overline{Y}_{\alpha} \subset H^{m-k}(A,\bC)$ is an $L_{\alpha,\bC}$-module;
it is irreducible as a $G'_{\alpha,\bC}$-module, and has highest weight $\mu_k$.
It contains the $(m-k,0)$-form
\[
	w_k := \overline{u}_m \wedge \dots \wedge \overline{u}_{k+1},
\]
as well as the $(0,m-k)$-form
\[
	w'_k := g(w_k) = \overline{u}_1 \wedge \dots \wedge \overline{u}_{m-k}.
\]
Let
\[
	k' =
		\begin{cases}
			k, & \text{if } k < \frac{m}{2}; \\
			m-k, & \text{if } k >\frac{m}{2}.
		\end{cases}
\]
Thus for each $k = 1$, \dots, $m-1$, we have an $SL_m(\bC)$-irreducible module $V_{\alpha, k}$ in $\bigwedge^{k'} V_{\alpha, \bC}$, such that $V_{\alpha, k}$ contains a nonzero $(k',0)$ form $w_k$ and a nonzero $(0,k')$-form $w'_k$, and, the highest weight of $SL_m(\bC)$ on $V_{\alpha, k}$ is $\mu_k$.
Note that in each case $w_k$ is a vector of highest weight, while $w'_k$ is a vector of lowest weight.
Observe that the set
$\, \left\{ V_{\alpha, k} \mid \alpha \in S, \ 1 \leq k < m \right\} \,$
is invariant under the action of $\Aut (\bC)$.

Let $j$, $k$ be positive integers with $1 \leq k < m$.
Then $S^j V_{\alpha, k}$, the symmetric tensors on $V_{\alpha, k}$, give a represention of $SL_m(\bC)$ with highest weight $j\mu_k$, and highest weight vector $(w_k)^j$.
Let $V^j_{\alpha,k}$ be the $SL_m(\bC)$-module generated by $(w_k)^j$.
The highest weight vector in $V^j_{\alpha, k}$ is $(w_k)^j$ which is a $(jk',0)$-form.
The lowest weight vector in $V^j_{\alpha, k}$ is $g((w_k)^j) = (w'_k)^j$ which is a $(0, jk')$-form.
Thus $V^j_{\alpha, k}$ is an irreducible representation with highest weight $j\mu_k$ which contains both the  $(jk',0)$-form $(w_k)^j$ and the $(0,jk')$-form $(w'_k)^j$.
Observe that the set
\[
	\left\{ \, V^j_{\alpha, k} \mid \alpha \in S, \ 1 \leq k < m, \ j > 0 \, \right\}
\]
is invariant under the action of $\Aut (\bC)$.

Any irreducible representation $\pi$ of $SL_{m}(\bC)$ has highest weight
\[
	\mu = a_1\mu_1 + \dots + a_{m-1}\mu_{m-1}
\]
where the $a_j$ are nonnegative integers. Let
\[
	a = \sum_{k=1}^{m-1} k'a_k, \qquad b = \sum_{k=1}^{m-1} a_k.
\]
Then the representation $\bigotimes_{k=1}^{m-1} S^{a_k} {V_{\alpha, k}} \subset \bigwedge^a V_{\alpha, \bC}^b$ has highest weight $\mu$.
The vector $v_{\mu} := \bigotimes_{k=1}^{m-1} (w_k)^{a_k}$ generates an irreducible submodule $V_{\mu}^{\alpha}$ which has highest weight $\mu$.
Note that $V_{\mu}^{\alpha}$ contains both the nonzero $(a,0)$-form  $v_{\mu}$
and the nonzero $(0,a)$-form $g(v_{\mu})$.
Observe that the set
\[
	\bigg\{ \, V_{\mu}^{\alpha} \ \big\vert \ \alpha \in S, \ 
	\mu = \sum_{i=1}^{m-1} a_i \mu_i, \ a_i \geq 0 \,  \bigg\}
\]
is invariant under the action of $\Aut (\bC)$.

Any irreducible representation $\rho$ of $G'(\bC)$ is of the form
$\rho = \bigotimes_{\alpha \in S} \pi_{\alpha}$, where $ \pi_{\alpha}$ is an irreducible representation of $G'_{\alpha, \bC} \cong SL_m(\bC)$.
Let $V_{\rho} = \bigotimes_{\alpha \in S} V_{\pi_\alpha}$.
Then $V_{\rho}$ is an irreducible submodule of some $H^c(A^d, \bC)$ on which $G'_{\alpha, \bC}$ acts as $\rho$, and which contains both nonzero $(c,0)$-forms and nonzero $(0,c)$-forms.
Observe that the set
\[
	\left\{ V_{\rho} \mid \rho \text{ a nontrivial irreducible representation of } G'(\bC)  \right\}
\]
is invariant under the action of $\Aut (\bC)$, so every Galois conjugate of $V_{\rho}$ contains a nonzero $(c,0)$-form.
\end{proof}

\begin{theorem}
\label{weakimpliesdomination}
Let $A$ be a weakly self-dominated abelian variety of \pel-type, such that each simple factor of $A$ is of type \textup{IV}.
Then, $A$ is dominated by the set of abelian varieties of the form $A^n \times B$, where $n$ is a positive integer, and $B$ is a product of CM abelian varieties with CM by subfields of $D(A)$.
\end{theorem}

\begin{proof}
We may assume that $A = A_1^{n_1} \times A_2^{n_2} \times \dots \times A_{\ell}^{n_{\ell}}$ where the $A_i$ are pairwise nonisogenous abelian varieties.
Let $D_i = D(A_i)$, $E_i$ the center of $D_i$, $F_i$ the maximal totally real subfield of $E_i$, $S_i$ the set of embeddings of $F_i$ into $\bR$, $V_i = H^1(A_i, \bQ)$, and, $m_i = \dim_{E_i} V_i$.
For each $i$ we have $V_{i,\bR} = \bigoplus_{\alpha \in S_i} V_{\alpha}$, and, $V_{\alpha, \bC} = Y_{\alpha} \oplus \overline{Y}_{\alpha}$ as in the proof of the previous theorem.
Let $S$ be the disjoint union of the sets $S_i$.
We then have $L(A)_{\bR} = \prod_{\alpha \in S} L_{\alpha}$, where
$L_{\alpha} = U(p_{\alpha}, q_{\alpha})$.

Let $W_i = \bigwedge_{E_i}^{m_i} H^1(A_i, \bQ)$ be the Weil Hodge structure in $H^{m_i}(A_i,\bQ)$ (see \cite{MoonenZarhin98}).
Let
\begin{equation}
\label{w}
	W = \bigoplus_{i=1}^{\ell} W_i.
\end{equation}
Then $W_{\bC} = \bigoplus_{\alpha \in S} W_{\alpha}$, where $W_{\alpha} = \bigwedge^{m_i }Y_{\alpha} \oplus \bigwedge^{m_i} \overline{Y}_{\alpha}$ for $\alpha \in S_i$.
$\bigwedge^{m_i} Y_{\alpha}$ is of Hodge type $(p_{\alpha},q_{\alpha})$ and $\bigwedge^{m_i} \overline{Y}_{\alpha}$ is of Hodge type $(q_{\alpha},p_{\alpha})$.
We note that $G'(A)$ acts trivially on $W$, so the Hodge group of $W$ is abelian and, therefore, $W$ is a Hodge structure of CM-type.

Let $P_{\alpha} \colon L(\bC) = \prod_{\beta \in S} L_{\beta, \bC} \to L_{\alpha, \bC}$
be the projection.
For $g \in L(\bC)$, let $\det_{\alpha}(g) = \det P_{\alpha} (g)$.

Let $V$ be an irreducible Hodge substructure of $H^b (A^d, \bQ)$ for some~$b$, $d$.
If $G'(A)$ acts trivially on $V$, then $V$ is of CM-type, so by \cite{Abdulali2005}*{Theorem~3, p.~159} there exists an abelian variety $B$ of CM-type and a fully twisted Hodge structure $V'$ in the cohomology of $B$ such that $V'$ is isomorphic to a Tate twist of $V$.
Suppose next that $G'(A)$ acts nontrivially on $V$.
Let $U$ be an irreducible $G(A)_{\bC}$-submodule of $V_{\bC}$ and denote by $\rho$ the action of $G'(A)_{\bC}$ on $U$.
Since $A$ is weakly self-dominated there exists an irreducible $L(A)_{\bC}$-submodule $V_{\rho}$ of $H^{c_{\rho}}(A^{d_{\rho}}, \bC)$ satisfying the conditions of Definition \ref{selfdominated}.
Then, as a $G(A)_{\bC}$-module, $U$ is equivalent to  $V_{\rho} \otimes \chi$, where $\chi$ is a character of the form $\chi = \bigotimes_{\alpha \in S} \det_{\alpha}^{n_{\alpha}}$.
The character $\chi$ occurs in the tensor algebra of $W$.
Let $Z$ be an irreducible Hodge structure in the tensor algebra of $W$ such that
$Z_{\bC}$ contains an irreducible submodule $W_{\chi}$ on which $L(A)_{\bC}$ acts as the character $\chi$.

By the main theorem of \cite{Abdulali2005} (Theorem~3, p.~159), there exist an abelian variety $B$ of CM-type
and an irreducible Hodge structure $Z' \subset H^c (B, \bQ)$ such that
$Z'$ is isomorphic to a Tate twist $Z(w)$ of $Z$, and, $Z'$ is fully twisted.
Let $\varphi \colon Z \to Z'$ be an equivalence of Hodge structures.
Let $Z'_{\chi} = \varphi(W_{\chi})$.
Then there exists $\sigma \in \Aut(\bC)$ such that $(Z'_{\chi})^{\sigma}$ contains
a nonzero $(c,0)$-form.
Let $U' = V_{\rho} \otimes Z'_{\chi} \subset H^{c_{\rho}+c} (A \times B, \bC)$.
Then $U'^{\sigma}$ contains a nonzero $(c_{\rho}+c, 0)$-form.
Let $\tU'$ be the smallest Hodge structure such that $U' \subset \tU'_{\bC}$.
Then $\tU'$ is a primary $G(A \times B)$-module.
Any irreducible submodule $V'$ of $\tU'$ is fully twisted and isomorphic to a Tate twist of $V$.
\end{proof}

\begin{remark}
\label{weakimpliesdominationremark}
In the above situation, let $\sB$ be a set of abelian varieties such that given any irreducible Hodge structure $Z$ in the tensor algebra of $W$ \eqref{w}, there exists a fully twisted Hodge structure $Z'$ in the cohomology of some $B \in \sB$, such that $Z'$ is isomorphic to a Tate twist of $Z$.
Then, the proof of Theorem~\ref{weakimpliesdomination} shows that $A$ is dominated by abelian varieties of the form $A^n \times B$, where $n$ is a positive integer, and, $B \in \sB$.
\end{remark}

Combining the previous results we get the following theorem.

\begin{theorem}
\label{maintheorem}
Let $A$ be an abelian variety  of \pel-type. Assume that each simple factor of $A$ is of type \textup{IV}.
Suppose $G'(A)(\bR) \cong \prod_{\alpha \in S} SU(p_{\alpha}, q_{\alpha})$, where
$\vert p_{\alpha} - q_{\alpha} \vert = 1$ for all $\alpha$.
Then any power of $A$ is dominated by the set of abelian varieties of the form $A^n \times B$, where $n$ is a positive integer, and $B$ is an abelian variety of CM-type.
\end{theorem}

\section{The General Hodge Conjecture}

We now apply the results of the previous section to deduce the general Hodge conjecture for \emph{products} of some of the abelian varieties for which we proved the general Hodge conjecture in \cite{Abdulali2004}.

\begin{theorem}
\label{ghc}
Let $\sA$ be the class of abelian varieties of \pel-type which are isogenous to products of abelian varieties of the following types:
\begin{enumerate}
\item
	a simple $3$-dimensional abelian variety with endomorphism algebra $\bQ (\sqrt{-1})$,
	with a polarization given by a hermitian form of signature $(2,1)$;
\item
	a simple $5$-dimensional abelian variety with endomorphism algebra $\bQ (\sqrt{-1})$,
	with a polarization given by a hermitian form of signature $(3,2)$;
\item
	an elliptic curve with CM by $\bQ (\sqrt{-1})$;
\item
	a simple $3$-dimensional abelian variety with endomorphism algebra $\bQ (\sqrt{-3})$,
	with a polarization given by a hermitian form of signature $(2,1)$;
\item
	a simple $5$-dimensional abelian variety with endomorphism algebra $\bQ (\sqrt{-3})$,
	with a polarization given by a hermitian form of signature $(3,2)$;
\item
	an elliptic curve with CM by $\bQ (\sqrt{-3})$;
\end{enumerate}
Then, any $A \in \sA$ is dominated by $\sA$, and the general Hodge conjecture holds for all members of $\sA$.
\end{theorem}

\begin{proof}
Let $A \in \sA$. Up to isogeny, $A = A_1 \times A_3$, where each simple factor of $A_1$ has endomorphism algebra $\bQ (\sqrt{-1})$, and each simple factor of $A_3$ has endomorphism algebra $\bQ (\sqrt{-3})$. We may assume that
\[
	A_1 = B_1^{m_1} \times \dots \times B_s^{m_s} \times E_1^m,
\]
where the $B_i$ are pairwise nonisogenous simple abelian varieties of dimension $3$ or $5$, and, $E_1$ is the elliptic curve with CM by $\bQ (\sqrt{-1})$.
Also,
\[
	A_3 = C_1^{n_1} \times \dots \times C_t^{n_t} \times E_3^n,
\]
where the $C_j$ are pairwise nonisogenous simple abelian varieties of dimension $3$ or $5$, and, $E_3$ is the elliptic curve with CM by $\bQ (\sqrt{-3})$.

We have shown in \cite{Abdulali2004}*{p.~208} that
$G(B_i \times E_1) = G'(B_i) \times G(E_1)$
and $G(C_j \times E_3) = G'(C_j) \times G(E_3)$.
It follows that if $m$ and $n$ are positive, then,
\[
G(A) = G'(A) \times G(E_1) \times G(E_3) = G(A_1) \times G(A_3).
\]

By Theorem \ref{maintheorem}, $A$ is dominated by abelian varieties of the form $A^n \times B$, where $B$ is an abelian variety of CM type.
Since $\bQ (\sqrt{-1})$ and $\bQ (\sqrt{-3})$ are linearly disjoint,
Remark~\ref{weakimpliesdominationremark} and \cite{Abdulali2005}*{Proposition~5, p.~160}
show that $B$ may be taken to be of the form $E_1^i \times E_3^j$.
Thus $A$ is dominated by $\sA$, and,
the usual Hodge conjecture for all members of $\sA$ implies the general Hodge conjecture for the same class.

Since $A = A_1 \times A_3$, with $G(A) = G(A_1) \times G(A_3)$,
the usual Hodge conjecture for each of $A_1$ and $A_3$ implies
the usual Hodge conjecture for $A$.
Let $X$ be one of $A_1$ or $A_3$.
Let $K$ be the endomorphism algebra of a simple factor of $X$, and let $E$ be the elliptic curve with CM by $K$.
Write
\[
	X = X_1^{k_1} \times \dots \times X_r^{k_r} \times E^{\ell},
\]
where the $X_i$ are pairwise nonisogenous simple abelian varieties of dimension $3$ or $5$, and
assume without loss of generality that $\ell > 0$. 
Then $G(X) = G'(X) \times G(E)$.
We shall prove the usual Hodge conjecture for $X$ by induction on $r$, the case $r = 1$ being Corollary 3.3 of \cite{Abdulali2004}.
For $r > 1$, let
\[
	\overline{X} = X_1^{k_1} \times \dots \times X_{r-1}^{k_{r-1}} \times E^{\ell},
\]
so that $X = \overline{X} \times X_r^{k_r}$,
and assume the usual Hodge conjecture holds for $\overline{X}$.

The Hodge ring of $X$ is given by
\begin{align*}
	H^\star(X, \bQ)^{G(X)} &=
	\bigoplus_{a,b} \left( H^a(\overline{X}, \bQ) \otimes H^b(X_r^{k_r},\bQ) \right)^{G(X)} \\
	&= \bigoplus_{c,b} \hom_{G(X)} \left( H^c (\overline{X}, \bQ),H^b(X_r^{k_r},\bQ) \right).
\end{align*}
Thus the Hodge ring of $X$ is generated by equivalences between Hodge substructures of the cohomology rings of $\overline{X}$ and $X_r^{k_r}$.

Let $W$ be the Weil Hodge structure in the cohomology of $X_r$.
Then $G(E)_{\bC}$ acts on $W_{\bC}$ as $\det \oplus \det^{-1}$.
Let $d$ be a positive integer.
Denote by $W_d$ the Hodge structure in the cohomology of $X_r^d$ on which $G(E)_{\bC}$ acts as $\det^d \oplus \det^{-d}$.
Similarly, denote by $W'_d$ the Hodge structure in the cohomology of $E^d$,
such that $G(E)_{\bC}$ acts as $\det^d \oplus \det^{-d}$ on $W'_{d,\bC}$

Let $U$ and $U'$ be isomorphic irreducible $G(X)$-submodules of $H^b(X_r^{k_r},\bQ)$
and $H^c (\overline{X}, \bQ)$ respectively.
Then $G'(X)$ acts trivially on $U$ and $U'$.
Thus every irreducible $G(X)_{\bC}$-submodule of $U_{\bC}$
is equivalent to $\det^a$ for some $a$.
Hence $U$ is equivalent to the Hodge structure $W_d$ for some $d \ge 0$.
Since $\dim X_r$ is a prime, the usual Hodge conjecture holds for all powers of $X_r$
by \cite{Ribet}*{Theorem 2}, and the equivalence of $U$ with $W_d$ is induced by an algebraic cycle.
Similarly, $U'$ is equivalent to $W'_d$.
By our induction hypothesis, the usual Hodge conjecture holds for $\overline{X}$, so the equivalence of $U'$ with $W'_d$ is induced by an algebraic cycle.
Since the usual Hodge conjecture is known for all powers of $X_r \times E$ (see \cite{Abdulali2004}), the equivalence of $W_d$ and $W'_d$ is also given by an algebraic cycle.
Thus the equivalence between $U$ and $U'$ is induced by an algebraic cycle.
It follows that $\hom_{G(X)} \left( H^c (\overline{X}, \bQ),H^b(X_r^{m_r},\bQ) \right)$ is generated by algebraic cycles, proving the usual Hodge conjecture for $X$.
\end{proof}

\end{document}